\documentclass[12pt]{article}
\addtocounter{secnumdepth}{2}
\usepackage{amsmath,amsthm}
\usepackage{amssymb}
\usepackage{graphics}
\usepackage{graphicx}
\usepackage{epsfig}
\usepackage{fancyhdr}
\usepackage{bbm}
\usepackage{wasysym}
\usepackage{latexsym}
\newtheorem{lemma}{Lemma}

\newtheorem{corollary}{Corollary}

\usepackage{amsfonts}
\usepackage{subfigure}
\usepackage{color}
\usepackage[normalem]{ulem}
\begin{document}
\title{On Point Processes Defined by Angular Conditions
on Delaunay Neighbors in the Poisson-Voronoi Tessellation}
\author{Fran\c{c}ois Baccelli and Sanket S. Kalamkar}
\date{October 30, 2020}
\maketitle
\begin{abstract}
Consider a homogeneous Poisson point process of the Euclidean plane and its Voronoi tessellation.
The present note discusses the properties of two stationary point processes associated 
with the latter and depending on a parameter $\theta$. The first one is the set of points
that belong to some one-dimensional facet of the Voronoi tessellation and are such that
the angle with which they see the two nuclei defining the facet is $\theta$.
The main question of interest on this first point process is its intensity.
The second point process is that of the intersections of the said tessellation
with a straight line having a random orientation. Its intensity is well known.
The intersection points almost surely belong to one-dimensional facets.
The main question here is about the Palm distribution of the angle with which
the points of this second point process see the two nuclei associated with the facet.
The note gives answers to these two questions and briefly discusses their practical motivations. It also discusses natural extensions to dimension three.
\end{abstract}

\section{Introduction}

The statistical properties of the facets of 
the Voronoi tessellation of homogeneous point processes of the Euclidean plane 
are well-studied~\cite{Moller_1989}. 

This note is focused on a question which was apparently not considered so far,
which is the distribution of the angle with which points of
the one-dimensional facets of the Poisson-Voronoi tessellation see the two Delaunay neighbors defining the facet.
The motivation for this question stems from cellular radio networks and is briefly discussed
in the note. The problem is however of independent interest.

Let $\Phi = \lbrace X_1, X_2, \dotsc \rbrace$ be a homogeneous Poisson
point process of intensity $\lambda>0$ on $\mathbb{R}^2$.
Let $\mathcal{V}_{X_i} \in \mathbb{R}^2$ denote the Voronoi cell with nucleus $X_i \in \Phi$:
\begin{align}
\mathcal{V}_{X_i} := \left\lbrace x \in \mathbb{R}^2 : \|x - X_i\| \leq \inf_{X_j \in \Phi\setminus \lbrace X_i \rbrace} \|x - X_j\| \right\rbrace.
\end{align} 
It is well known that the Voronoi cells in question are all a.s. finite random polygons.  
The topological boundary of each cell consists in an a.s. finite number of a.s finite segments.
Each such segment is associated with a so called Delaunay pair, namely a pair of points of $\Phi$ 
such that the cells of these two points share a common boundary segment.

Consider the set of the points of the one-dimensional facets
of the Voronoi tessellation of $\Phi$ that see their Delaunay pair with a given angle.
This discrete set of points forms a stationary point process
which is a factor of the point process $\Phi$. It is discussed in Section \ref{sec:app}
where its intensity is determined. 

Another natural model features a random line of the plane and the
intersections of this line with the one-dimensional facets.
This defines a stationary point process on the line which is discussed in Section \ref{sec:lpp}.
The Palm distribution of the angles at which the points of the latter
point process see the Delaunay pairs of the facets that intersect the line is determined.

These problems have natural extensions in dimension three which are discussed
in Section \ref{sec:3d}.

Finally, Section \ref{sec:cellular} presents the cellular networking motivations
of the problems alluded to above.

\section{Planar Point Process with Prescribed Delaunay Angle}
\label{sec:app}

Below, an intrinsic order on pairs of points of $\mathbb{R}^2$ is selected, e.g., the natural
on the $x$ coordinate. 
For all pairs of points $(D,D')$ of $\mathbb{R}^2$ such that $D<D'$ w.r.t. this order,
and for all points $Z$ of $\mathbb{R}^2$, let $\widehat{D,Z,D'}$ denote the angle from $D$ to $D'$
in, e.g., the anti-trigonometric direction and in the referential with origin $Z$. 

Let $\theta\in (0,2\pi)$ be fixed. For each segment $S$ of the Voronoi tessellation of $\Phi$,
there is either 0, or 1 point $Z$ on this segment satisfying the following property: denote by $D_1$ and $D_2$ the Delaunay neighbors associated with $S$, ordered as above; there is either 0, or 1 point $Z$ of the segment such that the angle $\widehat{D_1,Z,D_2}$ is equal to $\theta$ mod $2\pi$.

Let $\Psi_\theta$ be the point process in $\mathbb R^2$ of all points satisfying the above property.
This is illustrated in Figure~\ref{fig:base_fig}, which depicts a point of the segment
belonging to the intersection of the boundary $\mathcal{V}_{X_1}$ and that of $\mathcal{V}_{X_2}$
satisfying this angular property.

\begin{figure}[!h]
\centering
\includegraphics[scale=1.2]{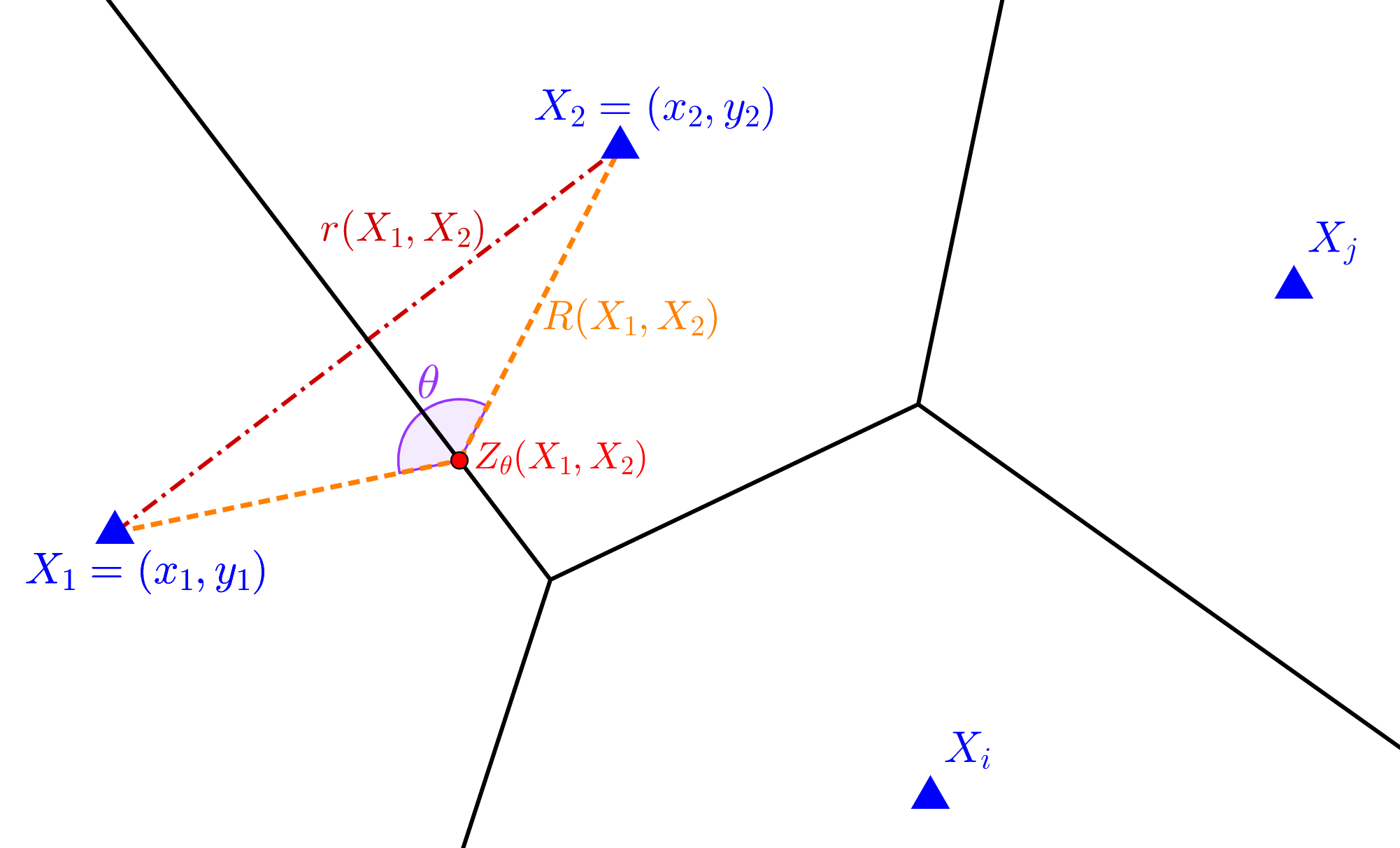}
\caption{
Triangles denote the points of the Poisson point process $\Phi$. Solid black lines denote the one-dimensional facets of the Poisson-Voronoi cells. The points $X_1$ and $X_2$ are the Delaunay neighbors associated with the one-dimensional facet containing $Z_{\theta}$, ordered as indicated above. Here $\theta\in (0,2\pi)$ with the above conventions.
}
\label{fig:base_fig}
\end{figure} 

\begin{lemma}
\label{lem1}
For all $\theta\in (0,2\pi)$, 
$\Psi_\theta$ is a stationary and ergodic point process.
Its intensity $\gamma_\theta$ is equal to $2\lambda \sin^2\frac{\theta}{2}$.
\end{lemma}

\begin{proof}
For all such $\theta$,
$\Psi_\theta$ is a factor point process of $\Phi$. It is hence stationary and mixing.

For all ordered points $X_1\ne X_2$ of $\Phi$, let $Z_{\theta}(X_1,X_2)$
be the point $Z$ that belongs to the bisector line of $(X_1,X_2)$ and is such that $\widehat{X_1,Z,X_2}=\theta$.
Let $B(x,r)$ be the open ball of center $x$ and radius $r$. The point $Z_{\theta}(X_1,X_2)$ 
is in the support of $\Psi_\theta$ if and only if
$\Phi(B(Z_{\theta}(X_1,X_2),\|Z_{\theta}(X_1,X_2)-X_1\|)=0$. 

Consider the following mass transport: send mass 1 from $X\in \Phi$ to $Z\in \Psi_\theta$
if there exists $Y\in \Phi$ such that $Y>X$, $Z=Z_\theta(X,Y)$ (so that $\widehat{X,Z,Y}=\theta$), and
$\Phi(B(Z,\|Z-Y\|)=0$.
Every point of $\Psi_\theta$ receives mass 1. Hence, by the mass transport principle,
\begin{eqnarray*}
\gamma_{\theta} & = & \lambda
\mathbb{P}_\Phi^0\left[\cup_{Y \in \Phi\setminus\{0\}}
\{\Phi\left(B\left(Z_{\theta}(0,Y), \|Z_{\theta}(0, Y)\|\right)\right) = 0\} \right]
\\ & = &
\mathbb{E}_\Phi^0\left[\sum_{Y \in \Phi\setminus\{0\}}
\boldsymbol{1}_{\Phi\left(B\left(Z_{\theta}(0,Y), \|Z_{\theta}(0, Y)\|\right)\right) = 0}
\right]
\\ & = &
\mathbb{E}\left[\sum_{Y \in \Phi}
\boldsymbol{1}_{\Phi\left(B\left(Z_{\theta}(0,Y), \|Z_{\theta}(0, Y)\|\right)\right) = 0}
\right],
\end{eqnarray*}
where $\mathbb{P}_\Phi^0$ denotes the Palm probability of $\Phi$ and where
the last equality follows from Slivnyak's theorem.
Using now Campbell's formula and
moving to polar coordinates, one gets 
\begin{align}
\gamma_{\theta} &= \pi \lambda^2 \int_{0}^{\infty}
\exp\left(-\lambda \pi R_{\theta, r}^2\right) r\mathrm{d}r ,
\end{align}
where the integration is only for polar angles from $-\pi/2$ to $\pi/2$ because
of the ordering assumption and
where $R_{\theta, r}=\frac {r}{2|\sin(\theta/2)|}$. It follows that
\begin{align}
\gamma_{\theta} &= \pi \lambda^2 \int_{0}^{\infty}
\exp\left(-\lambda \pi \frac{r^2}{4\sin^2 \frac{\theta}{2}}\right)
r\mathrm{d}r = 2\lambda \sin^2\frac{\theta}{2}.
\end{align}
\end{proof}

Here are a few direct corollaries of Lemma \ref{lem1}.
The first one is about the mean number of points of $\Psi_\theta$ in the
typical Voronoi cell:

\begin{corollary}
\label{cor01}
\begin{equation}
\mathbb{E}^0_\Phi[\Psi_\theta({\mathcal{V}}_0)]= 2 \sin^2\frac{\theta}{2}.
\end{equation}
\end{corollary}

\begin{proof}
Consider the following mass transport:
send mass 1 from each point $X$ of $\Phi$ to each point of
$\Psi_\theta$ belonging to $\mathcal{V}_{X}$.
The formula then follows from the mass transport principle.
\end{proof}

The second one is:

\begin{corollary}
For a two-dimensional Poisson-Voronoi tessellation,
the mean number of one-dimensional facets of the typical
cell that contain (resp. do not contain) the middle 
point of the line segment joining the Delaunay neighbors which define the facet is equal to 4 (resp. 2).
\end{corollary}

\begin{proof}
The result immediately follows from the last corollary and the fact
that the mean number of facets of the typical cell is equal to 6.
\end{proof}

\section{Distribution of Delaunay Angle on a Line}
\label{sec:lpp}

The setting is the same as above, with $\Phi$ a homogeneous Poisson point process
of intensity $\lambda$ on the Euclidean plane and its Voronoi tessellation.

Consider a straight line with a random orientation and distance to the origin, independent of $\Phi$.
Due to isotropy, one can assume that the line is the $x$-axis.
Let $\Upsilon \in \mathbb{R}$ denote the point process of Voronoi boundary crossings along this line.
The points of $\Upsilon$ are represented by ``crosses" in Figure~\ref{fig:da_fig}.
The linear intensity of $\Upsilon$ is well known to be $\mu = \frac{4\sqrt{\lambda}}{\pi}$
\cite{Muche_crossings}.
The point process is stationary (compatible with shifts along the $x$-axis).

Almost surely, each point of $\Upsilon$ belongs to a one-dimensional facet of the Voronoi
tessellation of $\Phi$. As such, one can associate to each point
$Z$ of $\Upsilon$ the two Delaunay neighbors $X_1(Z)$ and $X_2(Z)$
associated to the one-dimensional facet $Z$ belongs to.
These points are ordered using the above convention, with $X_1<X_2$. Let 
$$\Theta(Z)=\widehat{X_1(Z),Z,X_2(Z)}\in (0,2\pi).$$ 
These angles are depicted on Figure~\ref{fig:da_fig}.

By the same compatibility w.r.t. shifts along the $x$-axis,
the random variables $\{\Theta(Z)\}$ are marks of the point process $\Upsilon$. 
Thus the Palm distribution of $\Theta$ is well defined.
Note that this Palm distribution in question is w.r.t. the linear 
point process $\Upsilon$ rather than $\Phi$.

\begin{figure}[!h]
\centering
\includegraphics[scale=1.2]{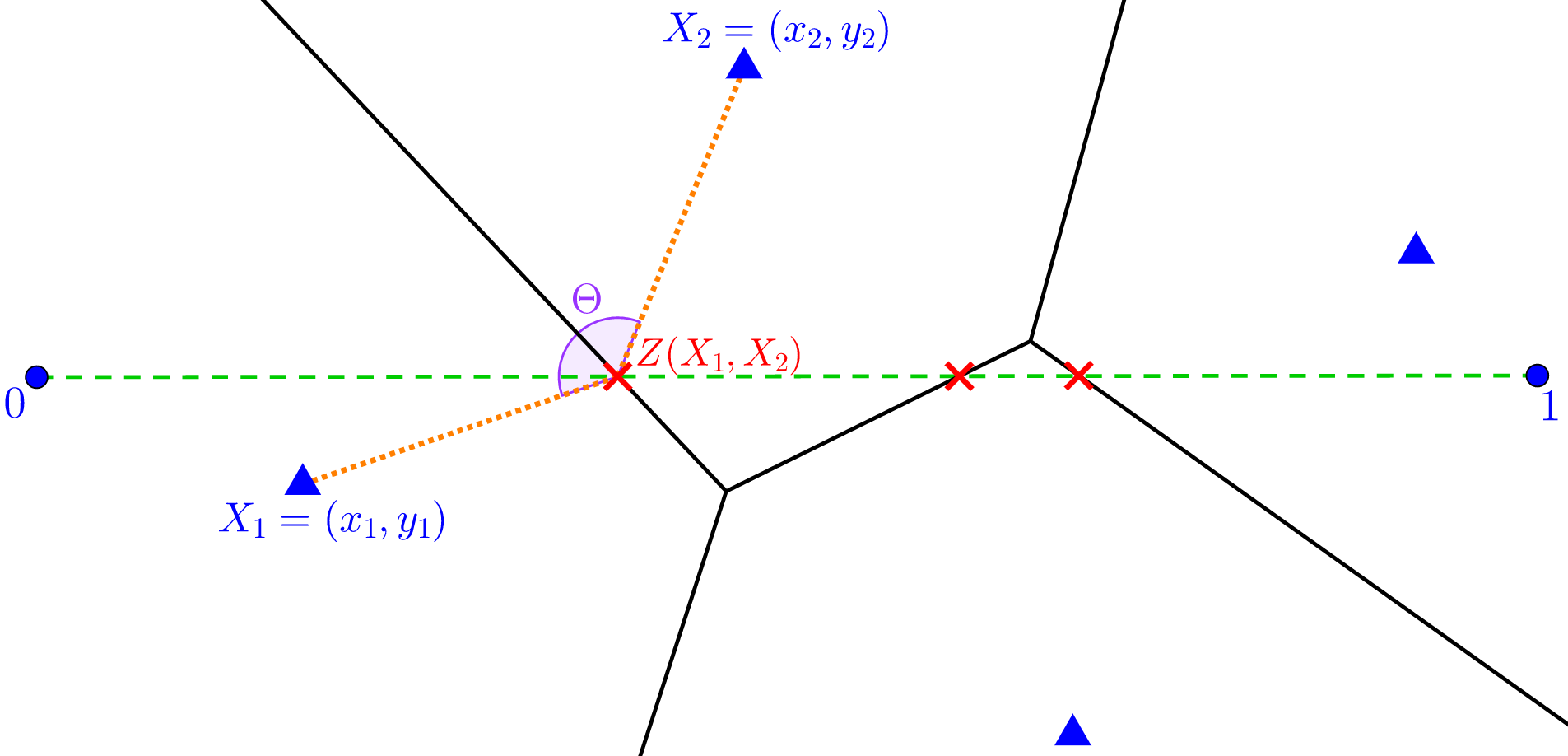}
\caption{Triangles denote the points of $\Phi$. Solid black lines denote the one-dimensional Voronoi facets. A cross denotes a point at the intersection of a Voronoi facet and the $x$-axis.}
\label{fig:da_fig}
\end{figure}

\begin{lemma}
\label{lem:lin}
The Palm distribution with respect to $\Upsilon$ of $\Theta=\Theta(0)$ has a density equal to
\begin{align}
f_{\Theta}(t) = \frac{1}{4} \sin \frac{t}{2}, \quad 0 < t < 2\pi.
\end{align}
\end{lemma}

\begin{proof}
Let $t$ be fixed with $t \in (0, \pi)$. 
Let $\Xi^+_t$ denote the sub-point process of $\Upsilon$ where only points
with an angular mark in $(t,\pi)$ are retained, that is
$$\Xi^+_t:= \sum_{Z\in \Upsilon} \delta_Z \mathbf{1}_{\Theta(Z)\in (t,\pi)}.$$
Since the selection of points of $\Upsilon$ which are retained to define $\Xi^+_t$ is based on marks,
the point process $\Xi^+_t$ is also stationary.
Let $\mu^+_{t}$ denote the (linear) intensity of $\Xi^+_t$.
By the definition of Palm probabilities, the two (linear) intensities 
$\mu^+_{t}$ and $\mu$ are related by the formula
$\mu^+_{t} = \mu \mathbb{P}_{\Upsilon}^{0}(\Theta(0) \in (t,\pi))$,
where $\mathbb{P}_{\Upsilon}^{0}(\cdot)$ denotes the Palm probability of $\Upsilon$. Thus
\begin{align}
\mathbb{P}_{\Upsilon}^{0}(\Theta(0) \in (t,\pi)) = \frac{\mu^+_{t}}{\mu}.
\end{align}

For all pairs $(X_1,X_2)$ of ordered points of $\Phi$, let $Z=Z(X_1,X_2)$ denote the intersection of the
bisector line of $(X_1,X_2)$ with the $x$-axis and $R=R(X_1,X_2)$ denote the
distance between $X_1$ and $Z(X_1,X_2)$. One has
\begin{eqnarray*}
\mu^+_{t} &  = &
\mathbb{E} \left[ \sum_{Z \in \Upsilon} \boldsymbol{1}_{Z\in [0,1]}
\boldsymbol{1}_{\Theta(Z) \in (t,\pi)} \right]
\nonumber \\
& = & \mathbb{E}\left[ \sum_{X_1 < X_2 \in \Phi}
\boldsymbol{1}_{Z(X_1,X_2)\in [0,1]}
\boldsymbol{1}_{Z(X_1,X_2) \in \mathcal{V}_{X_1}\cap \mathcal{V}_{X_2}}
\boldsymbol{1}_{\widehat{X_1,Z(X_1,X_2),X_2} \in (t,\pi)} \right].
\label{eq:three_ind}
\end{eqnarray*}
Using now the fact that the factorial moment measure of the Poisson point process of intensity
$\lambda$ is $\lambda^2 \mathrm{d}U_1\mathrm{d}U_2$,
where $\mathrm{d}U_i$, $i=1,2$ represents Lebesgue measure on $\mathbb{R}^2$, 
one gets that for all $t \in (0, \pi)$,
\begin{eqnarray*}
\mu_{t}^{+} &  = &
\lambda^2 \int_{U_1} \int_{U_2>U_1}
\mathbb{P}_{U_1,U_2}^0\left( Z(U_1,U_2) \in \mathcal{V}_{U_1}\cap \mathcal{V}_{U_2}\right)
\boldsymbol{1}_{Z(U_1,U_2)\in [0,1]}
\boldsymbol{1}_{\widehat{U_1,Z(U_1,U_2),U_2} \in (t,\pi)}  \mathrm{d}U_1 \mathrm{d}U_2,
\end{eqnarray*}
where $\mathbb{P}_{U_1,U_2}^0$ denotes the two point Palm probability of $\Phi$.
Let $U_1= (x_1, y_1)$ and $U_2 = (x_2, y_2)$.
The coordinates of $Z=Z(U_1,U_2)$ are
\begin{align}
Z = \left(\frac{1}{2}\frac{(x_2 - x_1)(x_2 + x_1) + (y_2 - y_1)(y_2 + y_1)}{x_2-x_1}, 0\right).
\label{eq:Z_coord}
\end{align}
Let also $R=R(U_1,U_2) = \|U_1 - Z\|$
and $r = \|U_1 - U_2\|$.
\begin{figure}[!h]
\centering
\includegraphics[scale=.8]{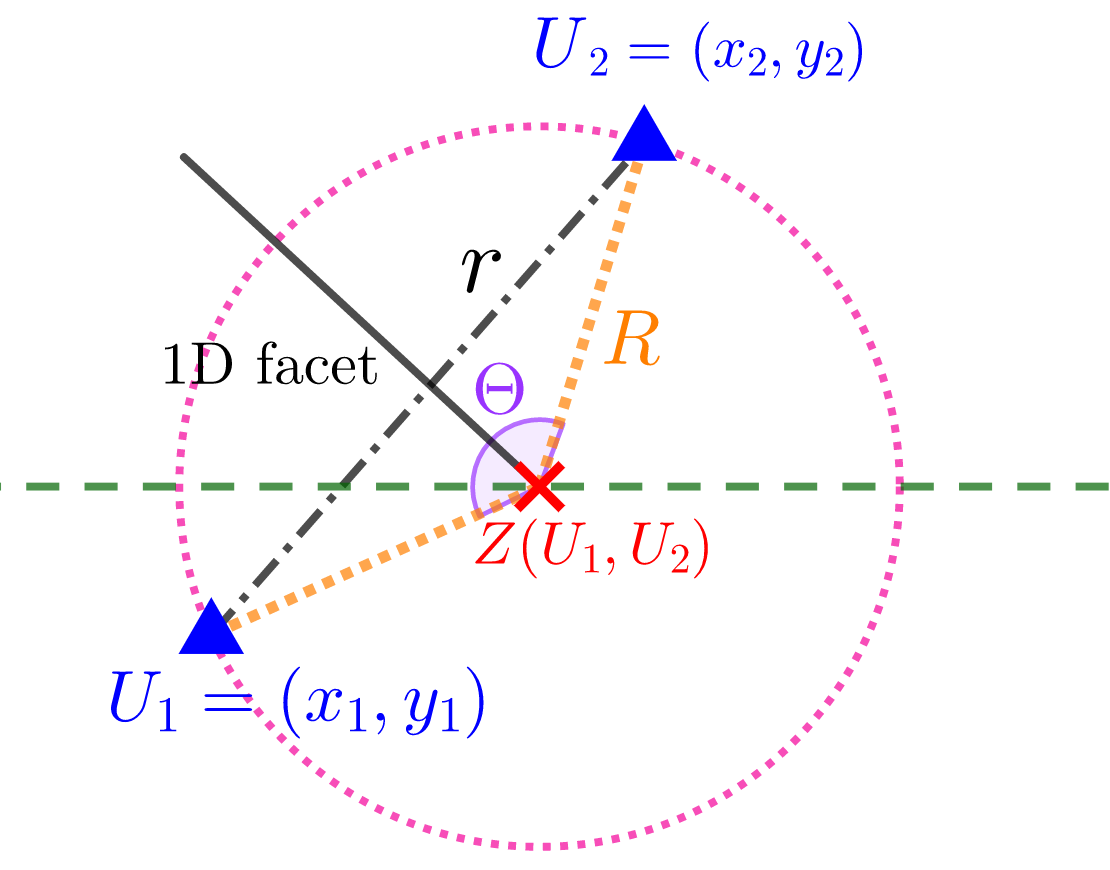}
\caption{$Z$ belongs to the one-dimensional facet of $U_1$ and $U_2$
if and only if there is no point of $\Phi$ within the open disk of radius $R$ that is centered at $Z$.}
\label{fig:null_fig}
\end{figure}

Using again the empty ball characterization of the Voronoi cell
(see Figure \ref{fig:null_fig}), one gets that for all $t\in(0,\pi)$,
\begin{eqnarray*}
\mu^+_{t} &  = & \frac 1 2
\lambda^2
\int_{U_1}\int_{U_2>U_1}\exp\left(-\lambda \pi R^2\right)
\boldsymbol{1}_{2\arcsin \frac{r}{2R} \in (t,\pi)}
\boldsymbol{1}_{Z \in [0, 1]}\mathrm{d}U_1 \mathrm{d}U_2,
\end{eqnarray*}
where the 1/2 comes from mirror symmetry w.r.t. $\pi$ and the fact that the integral (without the 1/2)
also counts the points $Z$ with an angle $\Theta$ in $(2\pi-t,2\pi)$.
So for all $t\in (0,\pi)$,
\begin{eqnarray*}
\mathbb{P}_{\Upsilon}^{0}(\Theta(0) \in (t,\pi)) 
 &  = & \frac 1 2
\frac {\lambda^2}{\mu}
\int_{U_1}\int_{U_2>U_1}\exp\left(-\lambda \pi R^2\right)
\boldsymbol{1}_{2\arcsin \frac{r}{2R} \in (t,\pi)}
\boldsymbol{1}_{{Z \in [0, 1]}}\mathrm{d}U_1 \mathrm{d}U_2\\
 &  = & \frac 1 4
\frac {\lambda^2}{\mu}
\int_{U_1}\int_{U_2}\exp\left(-\lambda \pi R^2\right)
\boldsymbol{1}_{2\arcsin \frac{r}{2R} \in (t,\pi)}
\boldsymbol{1}_{{Z \in [0, 1]}}\mathrm{d}U_1 \mathrm{d}U_2.
\end{eqnarray*}

The following stretch-rotation transformations are now used:
\begin{align*}
\begin{bmatrix} z_1 \\ z_2\end{bmatrix} = \begin{bmatrix} 1 & 1 \\ -1 & 1 \end{bmatrix} \begin{bmatrix} x_1 \\ x_2\end{bmatrix} \qquad \text{and} \qquad \begin{bmatrix} z_3 \\ z_4\end{bmatrix} = \begin{bmatrix} 1 & 1 \\ -1 & 1 \end{bmatrix} \begin{bmatrix} y_1 \\ y_2\end{bmatrix}.
\end{align*}

This yields ${\rm d}x_1 {\rm d}x_2 {\rm d}y_1 {\rm d}y_2=\frac{1}{4} {\rm d}z_1 {\rm d}z_2 {\rm d}z_3 {\rm d}z_4$. It follows that 
\begin{align*}
\mathbb{P}_{\Upsilon}^{0}(\Theta(0) \in (t,\pi)) &= \frac{\lambda^2}{16\mu} \int_{\mathbb{R}^4} \boldsymbol{1}_{\frac{1}{2}\left(z_1+\frac{z_3z_4}{z_2}\right) \in [0,1]} \times \boldsymbol{1}_{2\arcsin \sqrt{\frac{{z_2^2 + z_4^2}}{{z_2^2+z_4^2 + z_3^2\left(1 + \frac{z_4^2}{z_2^2}\right)}}} \in (t, \pi)} \nonumber \\
&\times e^{-\frac{\lambda \pi}{4}\left( z_2^2 + z_4^2 + z_3^2 \left(1+\frac{z_4^2}{z_2^2} \right) \right)} \, \mathrm{d}z_1\mathrm{d}z_2\mathrm{d}z_3\mathrm{d}z_4.
\end{align*}
Now the first indicator function can be eliminated by carrying out the integration
over $z_1$, where the condition for the first indicator function is met when 
\begin{align*}
0 \leq \frac{1}{2} \left(z_1+\frac{z_3z_4}{z_2}\right)
\leq 1, 
\end{align*}
or equivalently when $-\frac{z_3z_4}{z_2} \leq z_1 \leq 2- \frac{z_3z_4}{z_2}$,
which yields a factor of $2$. 

Polar coordinates, i.e., $(z_2,z_4)=(r\cos\phi,r\sin\phi)$
are now used to handle the second indicator function:
\begin{align}
\mathbb{P}_{\Upsilon}^{0}(\Theta(0) \in (t,\pi)) =\frac{\lambda^2}{8\mu}\int_{\mathbb{R}}{\rm d}z_3\int_0^\infty r\,{\rm d}r \int_0^{2\pi}{\rm d}\phi \, \boldsymbol{1}_{2\arcsin\frac{r}{\sqrt{r^2+\frac{z_3^2}{\cos^2\phi}}}\in (t, \pi)} \
e^{ -\frac{\lambda \pi}{4} \left( r^2 + \frac{z_3^2}{\cos^2\phi} \right)}.
\label{eq:int_z3}
\end{align}

The next step is to carry out the integration with respect to $z_3$,
which is twice the integral from $0$ to $\infty$.

Since $t<\pi$, the indicator function in the last integral
is equal to 1 on an interval with left limit $z_3^-=0$
(the argument of the arcsine is $1$ for $z_3=0$ and so the indicator is equal to
1 as $t \in (0,\pi)$) and with right limit obtained by solving 
\begin{align*}
\frac{r}{\sqrt{r^2+\frac{z_3^2}{\cos^2\phi}}}=\sin\frac{t}{2},
\end{align*}
namely $z_3^+=r|\cos \phi| \cot \frac{t}{2}$.
Hence, for $0<t<\pi$, one can write \eqref{eq:int_z3} as
\begin{align}
\mathbb{P}_{\Upsilon}^{0}(\Theta(0) \in (t,\pi))&= \frac{\lambda^{3/2}}{4\mu} \int_0^\infty r \, {\rm d}r \int_0^{2\pi} {\rm d}\phi \, |\cos\phi| \, \exp\left(-\frac{\lambda \pi r^2}{4}\right) \, {\rm erf}\left( \frac{r\sqrt{\lambda \pi} \, \cot \frac{t}{2}}{2} \right) \nonumber \\ 
&= \frac{\lambda^{3/2}}{\mu}  \int_0^\infty {\rm d}r \, r \, \exp\left(-\frac{\lambda \pi r^2}{4}\right) \, {\rm erf}\left( \frac{r\sqrt{\lambda \pi} \, \cot \frac{t}{2}}{2} \right) \nonumber \\ 
&= \frac{2\lambda \cot \frac{t}{2}}{\pi \mu} \int_0^\infty {\rm d}r \, \exp\left(-\frac{\lambda \pi r^2}{4\sin^2 \frac{t}{2}} \right) \nonumber \\
&=\frac{2 \sqrt{\lambda} \cos \frac{t}{2}}{\pi \mu} \nonumber \\
&\overset{(\rm a)}{=} \frac 1 2 \cos \frac{t}{2},
\label{eq:ccdf_ang}
\end{align}
where ${\rm erf}(x) = \frac{2}{\sqrt{\pi}} \int_{0}^{x} e^{-u^2}\mathrm{d}u$ is the error function
and $({\rm a})$ follows from the fact that $\mu = \frac{4 \sqrt{\lambda}}{\pi}$.

The pdf of $\Theta$ on $(0,t)$ follows by differentiating \eqref{eq:ccdf_ang}
with respect to $t$, which yields
\begin{align}
f_{\Theta}(t) = \frac{1}{4}\sin \frac{t}{2}, \quad t \in (0, \pi).
\label{eq:pdf_ang-2d}
\end{align}
The expression for the density of $\Theta$ in $(\pi,2\pi)$ follows by symmetry.
\end{proof}

\section{The Three-Dimensional Case}
\label{sec:3d}
In this section, $\Phi$ is a stationary Poisson point process of intensity $\lambda$
in $\mathbb{R}^3$ and $\nu(3)=\frac{4\pi}{3}$ denotes the volume of the unit sphere in dimension three.

\subsection{The subset of a facet seeing a given angle}

Consider a two-dimensional facet $F$ of the Voronoi tessellation of $\Phi$. 
Let $X_1$ and $X_2$ be the two nuclei creating $F$.
Let $C$ be the intersection point of the segment $[X_1,X_2]$ and the bisector
plane $P$ of this segment. Note that although $F\subset P$, $C$ does not necessarily belong to $F$.
For all $\theta\in (0,2\pi)$, the set ${\mathcal {Z}}_\theta(X_1,X_2)$ of points 
of $F$ that see the two nuclei $X_1$ and $X_2$ with angle $\theta$
is a random closed subset of $F$.
Here the angle is measured in the plane that contains $X_1,X_2$, and $Z$.
The set ${\mathcal {Z}}_\theta(X_1,X_2)$ is actually
the intersection of facet $F$ with the circle of center $C$ and radius $\rho$
in plane $P$, with
$$\rho:=\frac{||X_1-X_2||}{2} \left|\cot\frac {\theta}{ 2}\right|.$$

In the two-dimensional case, Corollary \ref{cor01} gives a formula
for the mean number of facets of the typical cell that contain
a point seeing the nucleus of the typical cell and the other nucleus defining the facet with angle $\theta$. 
More precisely, if $Z_\theta(X,Y)$ be the point of the bisector line of $[X,Y]$
that sees the pair $(X,Y)$ with angle $\theta\in (0,2\pi)$, then this corollary says that
\begin{equation}
\mathbb{E}^0_\Phi\left[\sum_{X\in \Phi,\ X\ne 0}
\boldsymbol{1}_{Z_\theta(0,X)\in \mathcal{V}_0} \right]=
2 \sin^2\frac{\theta}{2}.
\end{equation}
The three-dimensional analogue of the question considered 
in that corollary is about the Palm expectation of
the mean length, say $L_\theta$ of the set of loci of the facets of the typical Voronoi cell that
see the nucleus of this cell and the other nucleus defining the facet with angle $\theta$.
This analogue is evaluated in the following lemma:

\begin{lemma}
For all $\theta\in (0,2\pi)$, with $\theta \ne \pi$,
\begin{equation}
L_\theta:=\mathbb{E}^0_\Phi\left[\sum_{X\in \Phi,\ X\ne 0}l_1({\mathcal Z}_\theta(0,X))\right]=
4 \pi \left(\frac{6} {\pi \lambda}\right)^{\frac 1 3}
\Gamma\left(\frac 4 3\right)
\left|\cos\frac {\theta}{ 2}\right|
\sin^3\frac {\theta}{ 2},
\end{equation}
where $l_1$ denotes length.
\end{lemma}

\begin{proof}
By Slivnyak's theorem and Campbell's formula,
\begin{eqnarray}
\mathbb{E}^0_\Phi\left[\sum_{X\in \Phi,\ X\ne 0}l_1(Z_\theta(0,X))\right]
& = & 
\mathbb{E}\left[\sum_{X\in \Phi}l_1({\mathcal Z}_\theta((0,X),\Phi+\delta_0))\right]\nonumber \\
& = & \lambda 
\int_{x\in\mathbb{R}^3}
\boldsymbol{1}_{x<0} 
\int_{t=0}^{2\pi} \mathbb{P}_\Phi\left(Z_\theta(t,(0,x),\Phi+\delta_0+\delta_x)\in \mathcal{V}_0\right)
\mathrm{d} x
\mathrm{d} t,
\end{eqnarray}
where $Z_\theta(t,(0,x),\Phi+\delta_0+\delta_x)$
denotes the point of plane $P$ that is at the intersection of the circle of center $C$ and radius $\rho$
in this plane and the line of this plane containing $C$ and with direction $t$. 
Using now isotropy, one gets that
\begin{eqnarray}
\mathbb{E}^0_\Phi\left[\sum_{X\in \Phi,\ X\ne 0}l_1(Z_\theta(0,X))\right]
& = & 2\pi \lambda
\int_{x\in\mathbb{R}^3}
\boldsymbol{1}_{x<0} \rho(x)
\mathbb{P}_\Phi\left(Z_\theta(0,(0,x),\Phi+\delta_0+\delta_x)\in \mathcal{V}_0\right)
\mathrm{d} x\nonumber\\
& = & \pi \lambda
\int_{x\in\mathbb{R}^3}\rho(x)
\mathbb{P}_\Phi\left(Z_\theta(0,(0,x),\Phi+\delta_0+\delta_x)\in \mathcal{V}_0 \right)
\mathrm{d} x\nonumber\\
& = & \pi \lambda
\int_{x\in\mathbb{R}^3}\rho(x)
\exp(-\lambda \nu(3) R(x)^3) \mathrm{d} x,
\end{eqnarray}
with $\rho(x)$ defined as above, namely
$$\rho(x):=\frac{||x||}{2} \left|\cot\frac {\theta}{ 2}\right|,$$
and $R(x)$ the distance between $x$ and $Z$, namely
$$R(x):=\frac{||x||}{2} \frac{1} {\sin\frac {\theta}{ 2}}.$$
Passing to spherical coordinates, one gets
\begin{eqnarray}
\mathbb{E}^0_\Phi\left[\sum_{X\in \Phi,\ X\ne 0}l_1(Z_\theta(0,X))\right]
& = & 2 \pi^2 \lambda \left|\cot\frac {\theta}{ 2}\right|
\int_{r>0} \exp\left(-\lambda \pi \frac{r^3}{6} \frac{1} {\sin^3 
\frac {\theta}{ 2}}\right) r^3 \mathrm{d} r\nonumber\\
& = & 
4 \pi \left(\frac{6} {\pi \lambda}\right)^{\frac 1 3}
\Gamma\left(\frac 4 3\right)
\left|\cos\frac {\theta}{ 2}\right|
\sin^3\frac {\theta}{ 2}.
\end{eqnarray}
\end{proof}

The last result was for $\theta\ne \pi$.
For $\theta=\pi$, one should rather consider the point process $\Xi$ of points that belong to some facet 
and that are the middle points of the segment $[X_1,X_2]$ of the
Delaunay neighbors associated with this facet. The result is:

\begin{lemma}
\begin{equation}
N_\pi:= \mathbb{E}^0_\Phi\left[\Xi({\mathcal V}_0)\right]=8.
\end{equation}
\end{lemma}

\begin{proof}
By the same arguments as above, the mean number of points of the cell of 0 that
see the two ordered nuclei $(0,X)$ with angle $\theta=\pi$ is
\begin{eqnarray*}
M_\pi
& = & \frac {\lambda} {2}
\int_{x\in\mathbb{R}^3}
\exp\left(-\lambda \nu(3) \left(\frac{||x||}{2}\right)^3\right) \mathrm{d} x\\
& = & 2 \pi \lambda \int_{r>0} \exp\left(-\lambda \pi \frac 1 6 r^3\right) r^2 \mathrm{d} r\\
& = & 4.
\end{eqnarray*}
The result follows by symmetry.
\end{proof}

\subsection{The angles seen from a line}

The problem considered in Section \ref{sec:lpp} has a direct extension in dimension three,
where the question is again that of the Palm distribution of the angle $\Theta$ at which the 
intersections of the $x$-axis with the $2$-dimensional facets of the Voronoi tessellation of
a Poisson point process in $\mathbb{R}^3$ see the two nuclei creating the facet.

\begin{lemma}
\label{lem:lin3}
The Palm distribution with respect to $\Upsilon$ of $\Theta=\Theta(0)$ has the density
\begin{align}
f_{\Theta}(t) = \frac{3}{4} \left|\cos \frac{t}{2}\right| \sin^{2}\frac{t}{2}, \quad t\in (0, 2\pi).
\end{align}
\end{lemma}

\begin{proof}
By the same arguments as in the two-dimensional case, for all $t\in (0,\pi)$,
\begin{eqnarray*}
\mathbb{P}_{\Upsilon}^{0}(\Theta(0) \in (t,\pi)) 
&  = &
\frac {\lambda^2}{4 \mu(3)}
\int_{U_1\in \mathbb{R}^3}\int_{U_2\in \mathbb{R}^3}\exp\left(-\lambda \nu(3) R^3\right)
\boldsymbol{1}_{2\arcsin \frac{r}{2R} \in (t, \pi)}
\boldsymbol{1}_{{Z \in [0, 1]}}\mathrm{d}U_1 \mathrm{d}U_2,
\end{eqnarray*}
with $\mu(3) = \left(4\pi/3\right)^{1/3}\Gamma(5/3)\lambda^{1/3}$
the linear intensity of $2$-dimensional facet crossings, $\nu(3)$ the volume of the
unit sphere in dimension three, and $r$, $Z$ and $R$ geometrically defined as above. That is,
$r=\|U_1-U_2\|$, $Z$ is the point where the bisector plane of the line segment $[U_1,U_2]$
intersects the $x$-axis, and $R = \|Z-U_1\| = \|Z-U_2\|$.


Let $U_1 = (x_1, y_1, z_1)$ and $U_2 = (x_2, y_2, z_2)$.
The following stretch-rotation transformations are now used:
\begin{align*}
\begin{bmatrix} u_1 \\ u_2\end{bmatrix} = \begin{bmatrix} 1 & 1 \\ -1 & 1 \end{bmatrix} \begin{bmatrix} x_1 \\ x_2\end{bmatrix}~,~  \begin{bmatrix} u_3 \\ u_4\end{bmatrix} = \begin{bmatrix} 1 & 1 \\ -1 & 1 \end{bmatrix} \begin{bmatrix} y_1 \\ y_2\end{bmatrix} \qquad \text{and} \qquad  \begin{bmatrix} u_5 \\ u_6\end{bmatrix} = \begin{bmatrix} 1 & 1 \\ -1 & 1 \end{bmatrix}\begin{bmatrix} z_1 \\ z_2\end{bmatrix}.
\end{align*}

This yields ${\rm d}x_1 {\rm d}x_2 {\rm d}y_1 {\rm d}y_2 {\rm d}z_1 {\rm d}z_2=\frac{1}{8} {\rm d}u_1 {\rm d}u_2 {\rm d}u_3 {\rm d}u_4 {\rm d}u_5 {\rm d}u_6$. It follows that 
\begin{align*}
\mathbb{P}_{\Upsilon}^{0}(\Theta(0) \in (t,\pi))
&= \frac{\lambda^2}{32\mu(3)} \int_{\mathbb{R}^6}
\boldsymbol{1}_{\frac{1}{2}\left(u_1+\frac{u_3u_4}{u_2}+\frac{u_5u_6}{u_2}\right) \in [0,1]}
\times \boldsymbol{1}_{2\arcsin \sqrt{\frac{{u_2^2 + u_4^2 + u_6^2}}{u_2^2+u_3^2 + u_4^2 + u_5^2 + u_6^2 + \frac{u_3^2 u_4^2}{u_2^2}+ \frac{u_5^2 u_6^2}{u_2^2} + \frac{2 u_3u_4u_5u_6}{u_2^2}}} \in (t, \pi)} \nonumber \\
&\times e^{-\frac{A}{8} \left(u_2^2+u_3^2 + u_4^2 + u_5^2 + u_6^2 + \frac{u_3^2 u_4^2}{u_2^2}+ \frac{u_5^2 u_6^2}{u_2^2} + \frac{2 u_3u_4u_5u_6}{u_2^2} \right)^{3/2}}
\, \mathrm{d}u_1\mathrm{d}u_2\mathrm{d}u_3\mathrm{d}u_4\mathrm{d}u_5 \mathrm{d}u_6,
\end{align*}
where $A = 4\pi \lambda/3$. Now the first indicator function can be eliminated by carrying out the integration
over $u_1$, where the condition for the first indicator function is met when 
\begin{align*}
0 \leq \frac{1}{2} \left(u_1+\frac{u_3u_4}{u_2}+\frac{u_5u_6}{u_2}\right)
\leq 1, 
\end{align*}
or equivalently when $-\frac{u_3u_4}{u_2}-\frac{u_5u_6}{u_2} \leq u_1 \leq 2
- \frac{u_3u_4}{u_2} -\frac{u_5u_6}{u_2}$,
which yields a factor of $2$. 

Spherical coordinates, i.e., $(u_2, u_4, u_6) = (r \sin\psi \cos\varphi, r \sin\psi \sin\varphi, r \cos \psi)$
and polar coordinates, i.e., $(u_3,u_5)=(\rho\cos\phi,\rho\sin\phi)$
are now used to handle the second indicator function:
\begin{align}
\mathbb{P}_{\Upsilon}^{0}(\Theta(0) \in (t,\pi)) &=\frac{\lambda^2}{16\mu(3)}\int_{0}^{2\pi}{\rm d}\varphi\int_{0}^{\pi}{\rm d}\psi\sin\psi  \int_0^{2\pi}{\rm d}\phi\int_0^\infty {\rm d}r\,r^2 \nonumber \\
& \times \int_{0}^{\infty}\,\mathrm{d}\rho\,\rho \boldsymbol{1}_{2\arcsin\frac{r}{\sqrt{r^2+ a(\psi, \varphi, \phi)^2 \rho^2}}\in (t, \pi)}\,
e^{ -\frac{A}{8} \left(r^2+ a^2(\psi, \varphi, \phi)^2 \rho^2 \right)^{3/2}},
\label{eq:int_z3-3d}
\end{align}
where 
$$a^2(\psi, \varphi, \phi)^2 = \frac{1 - \left(\sin \psi \sin \phi \sin \varphi - \cos \psi \cos \phi \right)^2}{\sin^2 \psi \cos^2 \varphi}.$$

The next step is to carry out the integration with respect to $\rho$. The indicator function in question
is equal to 1 on an interval with left limit $\rho^-=0$
(the argument of the arcsine is $1$ and so $\pi>t$ is true as $t \in (0, \pi)$)
and with right limit obtained by solving 
\begin{align*}
{\frac{r^2}{r^2+a^2(\psi, \varphi, \phi)^2\rho^2  }=\sin^2\frac{t}{2}},
\end{align*}
namely $\rho^+=\frac{r}{a} \cot \frac{t}{2}$.
Hence, one can write the integrals w.r.t. $\rho$ and $r$ in \eqref{eq:int_z3-3d} as
\begin{align}
\int_0^\infty {\rm d}r  \, r^2 \int_0^{\frac{r}{a}\cot \frac{t}{2}} {\rm d}\rho \,
\rho \, e^{-\frac{A}{8}(a^2\rho^2+r^2)^{3/2}} =
\frac{32\,\Gamma(5/3)}{9a^2A^{5/3}} \left( 1 - \sin^3 \frac{t}{2}\right).
\end{align}
Thus the final step involves calculating the angular integrals as
\begin{align}
\mathbb{P}_{\Upsilon}^{0}(\Theta(0)\in (t,\pi))&= \frac{ 2\,\Gamma(5/3)\lambda^2}{9\,A^{5/3}\mu(3)} \left(1-\sin^3\frac{t}{2}\right) \int_0^{2\pi} {\rm d}\varphi \int_0^\pi {\rm d}\psi \, \sin\psi \int_0^{2\pi} {\rm d}\phi \, \frac{1}{a(\psi, \varphi, \phi)^2} \nonumber \\
&=\frac{2\,\Gamma(5/3)\lambda^2}{9\,A^{5/3}\mu(3)} \left(1-\sin^3\frac{t}{2}\right) \int_0^{2\pi} {\rm d}\varphi \int_0^\pi {\rm d}\psi \, \sin\psi \nonumber \\
&\times \int_0^{2\pi} {\rm d}\phi \, \frac{\sin^2\psi \cos^2\varphi}{1-\left(\sin\psi \sin\phi \sin\varphi - \cos\psi \cos\phi\right)^2}\nonumber  \\
&=\frac{4\pi\,\Gamma(5/3)\lambda^2}{9\,A^{5/3}\mu(3)} \left(1-\sin^3\frac{t}{2}\right) \int_0^{2\pi} {\rm d}\varphi \int_0^\pi {\rm d}\psi \, \sin^2\psi \, |\cos \varphi|\nonumber  \\
& \stackrel{A = \frac{4\pi \lambda}{3}}{=} \frac{1}{2}\left(\frac{4\pi}{3}\right)^{1/3}\frac{\Gamma(5/3)\lambda^{1/3}}{\mu(3)}\left(1-\sin^3\frac{t}{2}\right)\nonumber  \\
& =\frac 1 2 \left(1-\sin^3\frac{t}{2}\right)\,,
\label{eq:ccdf_ang-3d}
\end{align}
where the formula $\mu(3) = (4\pi/3)^{1/3}\Gamma(5/3)\lambda^{1/3}$ was used.

The pdf of $\Theta$ follows by differentiating \eqref{eq:ccdf_ang-3d} with respect to $t$, which yields
\begin{align}
f_{\Theta}(t) = \frac{3}{4}\cos \frac{t}{2} \sin^{2}\frac{t}{2},\quad t\in (0,\pi).
\end{align}
The expression for density of $\Theta$ follows from the symmetry and is given as 
\begin{align}
f_{\Theta}(t) = \frac{3}{4}\left|\cos \frac{t}{2}\right| \sin^{2}\frac{t}{2},\quad t\in (0, 2\pi).
\end{align}
\end{proof}
Note that the density is $0$ at $t = \pi$.

\section{Cellular Networking Motivations}
\label{sec:cellular}

Consider a cellular radio network where a mobile user connects to the nearest base station.
If the locations of the base stations are some realization of a stationary point process,
the service region of each base station is essentially the Voronoi cell of this base station
\cite{ABG_SG}. A mobile user moving on a straight line crosses cell boundaries of the Voronoi
tessellation, where it performs \textit{inter-cell handovers} \cite{FB_crossings} that involve
the transfer of the cellular connection between the two base stations sharing the cell boundary.
When the mobile user handset is equipped with two directional panels, the handset
might have to swap panels depending on whether the two base stations involved
in the inter-cell handover are seen by the same panel or not.
The event whether a panel swap occurs or not hence depends on the angle
at which the mobile user at the cell boundary sees the two base stations sharing the boundary.
The evaluation of the frequency of panel swaps at handover times requires evaluating 
the pdf of the angle with which the intersection point
of a randomly oriented line and the Voronoi facet sees the two nuclei that define the facet.

Assume that the base stations of a cellular radio network are located at positions that are
a realization of a Poisson point process $\Phi$ of intensity $\lambda$ in the plane as
illustrated in Figure \ref{fig:da_fig}.
The dashed straight line represents the path of a mobile user, which is assumed
to be along $x$-axis without loss of generality. A cross on the figure denotes a point
at the intersection of a Voronoi facet and the line of motion of the mobile user.
These points are those where the mobile user has to perform an inter-cell handover.

The situation motivating the previous analysis is that where the mobile
user is equipped with {\em directional panels}. The simplest situation is that where
the mobile user has two panels, one creating a beam covering the angular regions $[\chi,\chi+\pi)$, and the other a beam covering the region $[\chi+\pi,\chi+2\pi)$, where $\chi$ is uniformly distributed on $[0,\pi]$. When the mobile user reaches
an inter-cell handover point, two things may happen. If the two base stations
involved in the handover are not on the same side of the line with angle $\chi$
(i.e., are in the beams of different panels, as depicted on Figure \ref{fig:switch}),
there is a {\em panel swap}, which has a certain overhead cost.
If the base stations are on the same side of this line, there is no panel swap (see an instance
of this case on Figure \ref{fig:no_switch}) and no such cost is incurred by the mobile user.
In this context it is important to evaluate the ergodic fraction of 
inter-cell handovers that involve such a panel swap. 

The more general situation is that where there are $2^m$ panels with $m\ge 1$, each
surveying an angle (or beam) of the form $[\chi+2k\pi/2^m,\chi+2(k+1)\pi/2^m)$,
$k=0,1,\ldots, 2^m-1$. Here too, the main question is again about the fraction of
inter-cell handovers that involve a panel swap (which happens when the two base
stations are seen by the mobile user within different panel beams).
This question is answered in the next corollary of Lemma \ref{lem:lin}.

\begin{corollary}
When the typical mobile user has $2^m$ panels, $m\ge 1$,
the probability $p$ of a panel swap at the user during an inter-cell handover is 
\begin{align}
p = \frac{2^m}{\pi}\sin \frac{\pi}{2^m}, \quad m \geq 1.
\end{align}  
\end{corollary}

\begin{figure*}[t]
	\begin{minipage}{0.32\linewidth}
		\begin{center}
			\includegraphics[scale=0.7]{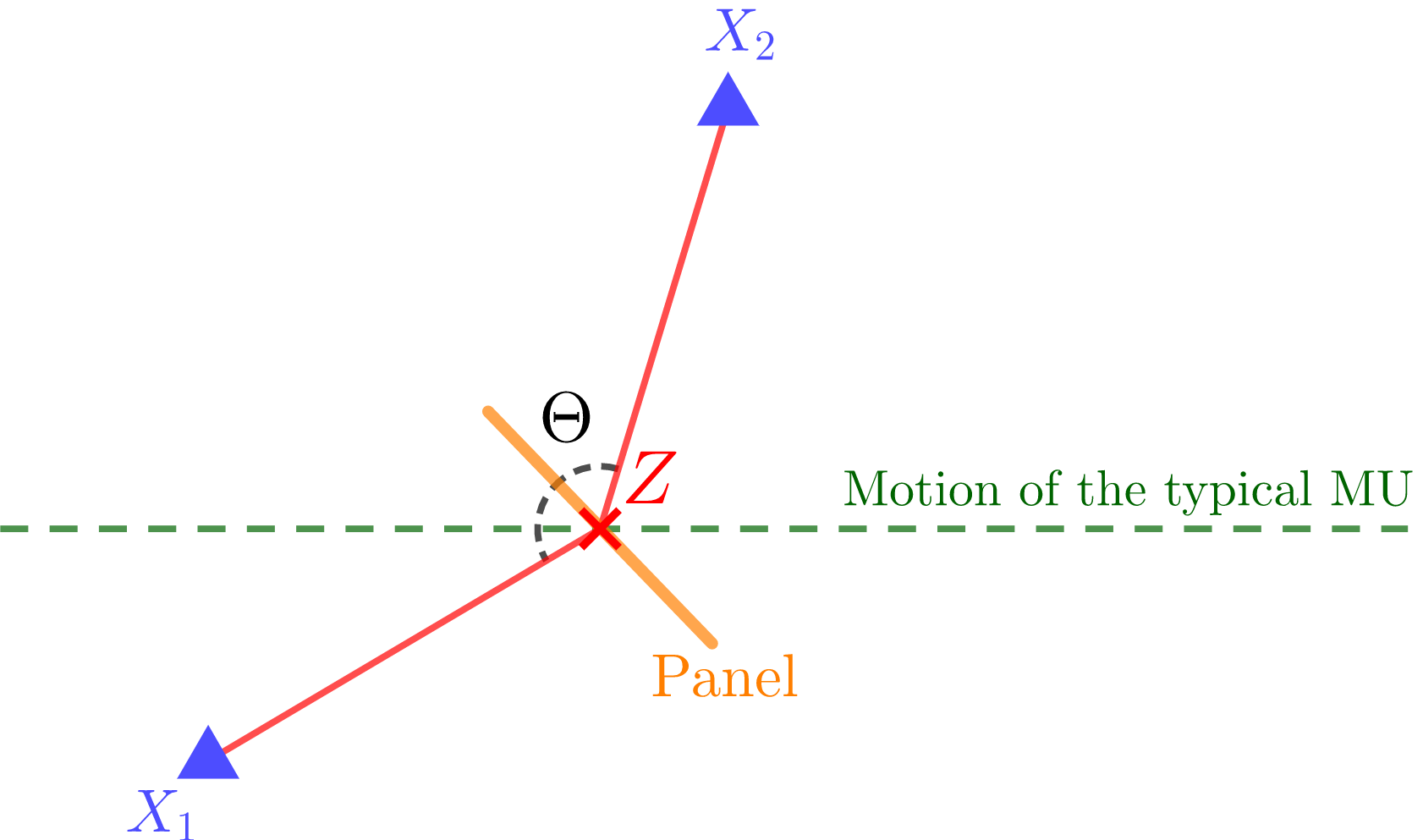}
			\caption{\footnotesize Panel swap with $2$ panels. }
			\label{fig:switch}
		\end{center}
	\end{minipage}
	\hspace{0.2\linewidth}
	\begin{minipage}{0.35\linewidth}
		\begin{center}
			\includegraphics[scale=0.7]{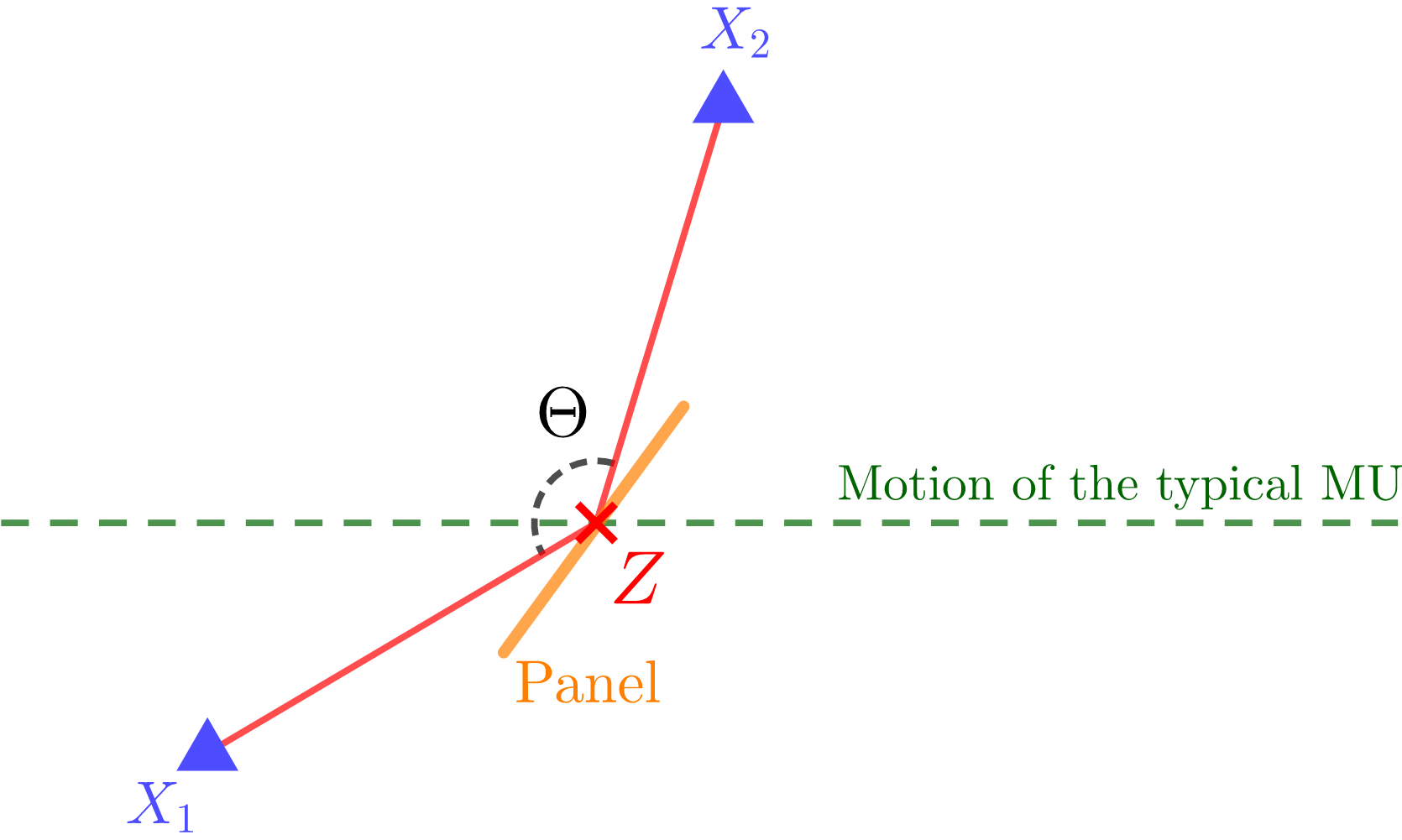}
			\caption{\footnotesize No panel swap with $2$ panels.}
			\label{fig:no_switch}
		\end{center}
	\end{minipage}
	\end{figure*}

\begin{proof}
The case of two panels, i.e., $m=1$, is considered first.
Without loss of generality, the coordinate system  can be taken such that the
inter-cell handover point is the origin $O$. Let $X$ denote the minimal Delaunay
neighbor, $A(X)$ its angle, and $\Theta$ the angle with which $O$ sees the two Delaunay neighbors, with
the foregoing conventions.

As Figures \ref{fig:switch} and \ref{fig:no_switch} show, a panel swap occurs if and only if
one of the two ends of the panel is ``within" the angle $\Theta$. More precisely, let
$\chi$ denote the angle of the panel, which is uniformly distributed on $(0, \pi)$.
If $\Theta\in (0,\pi)$, there is a panel swap if and only if
either $\chi\in (A(X),A(X)+\Theta)$ or $\chi+\pi \in (A(X),A(X)+\Theta)$,
and these two events cannot simultaneously hold.
Similarly, if $\Theta\in (\pi,2\pi)$, there is a panel swap if and only if
either $\chi\in (A(X)+\Theta, A(X))$ or $\chi+\pi \in (A(X)+\Theta,A(X))$,
with these two events excluding each other.
Since $\Theta$, $A(X)$ and $\chi$ are independent,
the probability of a panel swap is
\begin{eqnarray}
p &= &  
\int_0^\pi f_\Theta(t) \frac 1{2\pi}\left(\int_0^t \mathrm{d}u + \int_0^{t} \mathrm{d}u\right) \mathrm{d}t
+\int_\pi^{2\pi} f_\Theta(t) \frac 1{2\pi}\left(\int_0^{2\pi-t} \mathrm{d}u + \int_0^{2\pi-t} \mathrm{d}u\right) \mathrm{d}t
\nonumber \\
&= & \frac{1}{2\pi}\int_{0}^{\pi} t\sin{\frac t 2} \mathrm{d}t  \nonumber \\
&= & \frac{2}{\pi},
\end{eqnarray}
where the expression obtained in~\eqref{eq:pdf_ang-2d} was used.

This can be generalized to the case where the mobile user has $2^m$ panels, with $m\ge 1$.

Using the same notation as above, if $\Theta \in (0, \pi)$, a panel swap occurs if and only if
there is at least one $k = 0, 1, \dotsc, 2^m-1$ such that 
\begin{equation}
\label{eq:2tothem}
\chi + \frac{2k\pi}{2^{m}} \in (A(X), A(X) + \Theta).\end{equation}
If $\Theta \in (\pi, 2\pi)$, a panel swap occurs if and only if
\begin{equation}
\label{eq:2tothem2}
\chi + \frac{2k\pi}{2^{m}} \in (A(X) + \Theta, A(X)),
\end{equation}
for some $k= 0, 1, \dotsc, 2^m-1$.
Hence a panel swap is certain if
$\Theta \in \left(\frac{\pi}{2^{m-1}}, 2\pi-\frac{\pi}{2^{m-1}}\right)$.
The cases $\Theta \in \left(0, \frac{\pi}{2^{m-1}}\right)$
and $\Theta \in \left(2\pi-\frac{\pi}{2^{m-1}}, 2\pi \right)$ are symmetrical.
For $\Theta \in \left(0, \frac{\pi}{2^{m-1}}\right)$ 
(resp. $\Theta \in \left(2\pi-\frac{\pi}{2^{m-1}}, 2\pi\right)$),
there are $2^m$ symmetrical possibilities for a panel swap, one for each value of 
$k= 0, 1, \dotsc, 2^m-1$ in Equation (\ref{eq:2tothem}) (resp. (\ref{eq:2tothem2})). These events are disjoint
and have the same probability.
Since $\Theta$, $A(X)$, and $\chi$ are independent, the probability of a panel swap is hence
\begin{align}
p &= 2^{m+1}\int_{0}^{\frac{\pi}{2^{m-1}}}\mathrm{d}t f_{\Theta}(t)\frac{1}{2\pi} \int_{0}^{t}\mathrm{d}u + \int_{\frac{\pi}{2^{m-1}}}^{2\pi - \frac{\pi}{2^{m-1}}} f_{\Theta}(t)\mathrm{d}t \nonumber \\
&= \frac{2^m}{4\pi} \int_{0}^{\frac{\pi}{2^{m-1}}}t\sin \frac{t}{2} \mathrm{d}t + \frac{1}{4}\int_{\frac{\pi}{2^{m-1}}}^{2\pi - \frac{\pi}{2^{m-1}}} \sin \frac{t}{2}\mathrm{d}t \nonumber \\
& = \frac{2^m}{\pi} \left[\sin t - t \cos t\right]_0^{\frac{\pi}{2^m}}+ \cos \frac{\pi}{2^m} \nonumber \\
&= \frac{2^m}{\pi}\sin \frac{\pi}{2^m} - \cos \frac{\pi}{2^m} + \cos \frac{\pi}{2^m} \nonumber \\
&= \frac{2^m}{\pi}\sin \frac{\pi}{2^m}.
\end{align}

\end{proof}

\section*{Acknowledgements}
This work was  
supported by the ERC NEMO grant, under the European Union's Horizon 2020 research and innovation programme,
grant agreement number 788851 to INRIA.
The authors thank F. Abinader and L. Uzeda-Garcia of Bell Laboratories for the discussions that
led to the formulation of the problem. They also thank D. Stoyan, J. M{\o}ller, and M. Haenggi for their
comments on the topic discussed in this note.

\end{document}